\documentclass{amsart}
\usepackage[all]{xy}
\usepackage{verbatim}
\usepackage{color}
\usepackage{amsthm}
\usepackage{amssymb}
\usepackage[colorlinks=true]{hyperref}
\usepackage{graphicx}
\usepackage{mathtools}

\numberwithin{equation}{section}

\newtheorem{theorem}[equation]{Theorem}
\newtheorem*{theorem*}{Theorem} \newtheorem{lemma}[equation]{Lemma}

\newtheorem*{conjecture*}{Mamma Conjecture}
\newtheorem*{conjecture1*}{Mamma Conjecture (revisited)}
\newtheorem{proposition}[equation]{Proposition}

\newtheorem*{corollary*}{Corollary}

\theoremstyle{remark}

\newtheorem{example}[equation]{Example}

\theoremstyle{remark}
\newtheorem{remark}[equation]{Remark}

\newcommand{\cA}{{\mathcal A}}

\newcommand{\cD}{{\mathcal D}}
\newcommand{\cE}{{\mathcal E}}
\newcommand{\cF}{{\mathcal F}}
\newcommand{\cG}{{\mathcal G}}
\newcommand{\cH}{{\mathcal H}}

\newcommand{\cO}{{\mathcal O}}

\newcommand{\cT}{{\mathcal T}}

\newcommand{\Spt}{\mathrm{Spt}}

\newcommand{\bbA}{\mathbb{A}}

\newcommand{\bbZ}{\mathbb{Z}}

\DeclareMathOperator{\id}{id}

\newcommand{\dgcat}{\mathrm{dgcat}} 

\newcommand{\bbK}{I\mspace{-6.mu}K}

\newcommand{\perf}{\mathrm{perf}}

\newcommand{\dg}{\mathrm{dg}}

\newcommand{\dgHo}{\mathrm{H}^0}

\newcommand{\too}{\longrightarrow}

\newcommand{\ie}{\textsl{i.e.}\ }

\let\oldmarginpar\marginpar
\def\marginpar#1{\oldmarginpar{\tiny #1}}

\def\multiset#1#2{\ensuremath{\left(\kern-.3em\left(\genfrac{}{}{0pt}{}{#1}{#2}\right)\kern-.3em\right)}}

\begin{document}

\title[Vanishing of the negative KH of quotient singularities]{Vanishing of the negative homotopy $K$-theory \\of quotient singularities}
\author{Gon{\c c}alo~Tabuada}

\address{Gon{\c c}alo Tabuada, Department of Mathematics, MIT, Cambridge, MA 02139, USA}
\email{tabuada@math.mit.edu}
\urladdr{http://math.mit.edu/~tabuada}
\thanks{The author was partially supported by a NSF CAREER Award}

\date{\today}


\abstract{Making use of Gruson-Raynaud's technique of ``platification par \'eclatement'', Kerz and Strunk proved that the negative homotopy $K$-theory groups of a Noetherian scheme $X$ of Krull dimension $d$ vanish below $-d$. In this article, making use of noncommutative algebraic geometry, we improve this result in the case of quotient singularities by proving that the negative homotopy $K$-theory groups vanish below $-1$. Furthermore, in the case of cyclic quotient singularities, we provide an explicit ``upper bound'' for the first negative homotopy $K$-theory group.}
}

\maketitle
\vskip-\baselineskip
\vskip-\baselineskip

\section{Introduction and statement of results}\label{sec:introduction}
Given a Noetherian scheme $X$ of Krull dimension $d$, Kerz and Strunk proved in \cite[Thm.~1]{KS} that $KH_n(X)=0$ for every $n < -d$. The first goal of this article is to improve this vanishing result in the case where $X$ is a quotient singularity.

Let $k$ be a field of characteristic zero, $d\geq 2$ an integer, and $G$ a finite subgroup of $\mathrm{SL}_d(k)$. The group $G$ acts naturally on the polynomial ring $S:=k[t_1, \ldots, t_d]$ and the associated invariant ring $R:=S^G$ is Gorenstein and of Krull dimension $d$. Throughout the article we will assume that the local ring $R$ is an isolated singularity. 
\begin{example}[Cyclic quotient singularities]
Let $k$ be an algebraically closed field\footnote{More generally, it suffices that $k$ contains all the $m^{\mathrm{th}}$ roots of unit.} of characteristic zero, $m \geq 2$ an integer, $\zeta$ a primitive $m^{\mathrm{th}}$ root of unit, and $a_1, \ldots, a_d$ integers satisfying the following three conditions: $0 < a_j < m$, $\mathrm{gcd}(a_j,m)=1$ for every $1\leq j \leq d$, and $a_1 + \cdots + a_d=m$. When $G$ is the cyclic subgroup of $\mathrm{SL}_d(k)$ generated by $\mathrm{diag}(\zeta^{a_1}, \ldots, \zeta^{a_d})$, the ring $R$ is an isolated quotient singularity of Krull dimension $d$. For example, when $d=2$, $a_1=1$, and $a_2=m-1$, the ring $R$ identifies with the Kleinian singularity $k[u,v,w]/(u^m+vw)$ of type $A_{m-1}$.
\end{example}
\begin{remark}
As proved by Kurano-Nishi in \cite{KN}, all Gorenstein isolated quotient singularities of odd prime (Krull) dimension are cyclic. Moreover, in all the other (Krull) dimensions there exist {\em non-cyclic} Gorenstein isolated quotient singularities. For example, in dimension two all Gorenstein quotient singularities are isolated.
\end{remark}
Let us write $X$ for the affine (singular) $k$-scheme $\mathrm{Spec}(R)$. 
\begin{theorem}\label{thm:main1}
We have $KH_n(X)=0$ for every $n<-1$.
\end{theorem}
Intuitively speaking, Theorem \ref{thm:main1} shows that in the case of quotient singularities the vanishing of the negative homotopy $K$-theory groups is independent of the Krull dimension! To the best of the author's knowledge, this result is new in the literature. Such a vanishing result does not hold in general because, for every integer $d\geq 2$, there exist Gorenstein isolated singularities $X$ of Krull dimension $d$ with $KH_{-d}(X)\neq 0$; consult Reid \cite{Reid} for details. Theorem \ref{thm:main1} leads naturally to the following divisibility properties of nonconnective algebraic $K$-theory:
\begin{proposition}\label{prop:1}
\begin{itemize}
\item[(i)] The abelian group $\bbK_{-2}(X)$ is divisible;
\item[(ii)] The abelian groups $\bbK_n(X), n< -2$, are uniquely divisible.
\end{itemize}
\end{proposition}
\begin{remark}
Since the base field $k$ is of characteristic zero and the scheme $X$ is of finite type over $k$, it follows from the work of Corti\~nas-Haesemeyer-Schlichting-Weibel \cite{CHSW} that $\bbK_n(X)=0$ for every $n <-d$. Moreover, $\bbK_{-d}(X)\simeq KH_{-d}(X)$. Consequently, making use of Theorem \ref{thm:main1}, we conclude that~$\bbK_{-d}(X)=0$.
\end{remark}
The second goal of this article is to provide some information about $KH_{-1}(X)$ in the case of cyclic quotient singularities. Let $k$ be an algebraically closed field of characteristic zero and $(Q,\rho)$ the quiver with relations defined as follows:
\begin{itemize}
\item[(s1)] consider the quiver with vertices $\bbZ/m$ and with arrows $x^i_j\colon i \to i + a_j$, where $i \in \bbZ/m$ and $1 \leq j \leq d$. The relations $\rho$ are given by $x^{i+a_j}_{j'}x^i_j = x^{i+a_{j'}}_j x_{j'}^i$ for every $i \in \bbZ/m$ and $1\leq j,j' \leq d$.
\item[(s2)] remove from (s1) all arrows $x^i_j\colon i \to i'$ with $i > i'$;
\item[(s3)] remove from (s2) the vertex $0$.
\end{itemize}
Consider the $(m-1)\times (m-1)$ matrix $C$ such that $C_{ij}$ equals the number of arrows in $Q$ from $j$ to $i$ (counted modulo the relations). Let $M:=(-1)^{d-1}C(C^{-1})^T-\id$ and $\mathrm{M}\colon \oplus_{r=1}^{m-1} \bbZ \to \oplus_{r=1}^{m-1} \bbZ$ the associated (matrix) homomorphism.
\begin{theorem}\label{thm:main2}
The abelian group $KH_{-1}(X)$ is a quotient of the cokernel of $\mathrm{M}$.
\end{theorem}
Intuitively speaking, Theorem \ref{thm:main2} provides an explicit ``upper bound'' for the first negative homotopy $K$-theory group of cyclic quotient singularities.
\begin{example}[Kleinian singularities of type $A$]
When $d=2$, $a_1=1$, and $a_2=m-1$, the three steps (s1)-(s3) lead to the quiver $ Q\colon 1 \to 2 \to 3 \to \cdots \to m-2 \to m-1$ (without relations). Consequently, we obtain the following matrix:
$$
M= \begin{bmatrix}
-2 & 1 & 0  &\cdots &0 \\
-1 & -1 & \ddots &\ddots & \vdots\\
-1 & 0 & \ddots & \ddots & 0 \\
 \vdots& \vdots& \ddots & \ddots & 1 \\
 -1 &0 & \cdots& 0 & -1
\end{bmatrix}_{(m-1) \times (m-1)}
\,.$$
The cokernel of $\mathrm{M}$ is isomorphic to $\bbZ/m$; a generator is given by the image of the vector $(0,\ldots, 0,-1) \in \oplus_{r=1}^{m-1} \bbZ$. Thanks to Theorem \ref{thm:main2}, we hence conclude that $KH_{-1}(\mathrm{Spec}(k[u,v,w]/(u^m+vw)))$ is a quotient of $\bbZ/m$.
\end{example}

\begin{example}[A three dimensional singularity]
When $d=m=3$ and $a_1=a_2=a_3=1$, the three steps (s1)-(s3) lead to the quiver $Q\colon \xymatrix@C=1.7em@R=1em{1\ar@<0.7ex>[r]\ar[r]\ar@<-0.7ex>[r] & 2}$ (without relations). Consequently, we obtain the matrix $M=\begin{bmatrix} 0 & -3 \\3& -9\end{bmatrix}$. The cokernel of $\mathrm{M}$ is isomorphic to $\bbZ/3\times \bbZ/3$; generators are given by the images of $(1,0)$ and $(-1,-3)$. Thanks to Theorem \ref{thm:main2}, we hence conclude that $KH_{-1}(X)$ is a quotient of $\bbZ/3\times \bbZ/3$.
\end{example}
\subsection*{Preliminaries}
We will assume the reader is familiar with the language of differential graded (=dg) categories; consult Keller's ICM survey \cite{ICM-Keller}. Let $\dgcat(k)$ be the category of dg categories and dg functors. Given a (dg) $k$-algebra $A$, we will still write $A$ for the associated dg category with a single object. Consider the simplicial $k$-algebra $\Delta_n:=k[t_0, \ldots, t_n]/(\sum_i t_i -1), n \geq 0$, with faces and degenerancies: 
\begin{eqnarray*}
\partial_j(t_i) := \left\{ \begin{array}{lcr}
t_i & \text{if} & i <j \\
0 & \text{if} & i =j \\
t_{i-1} & \text{if} & i > j \\
\end{array} \right.
&
&
\delta_j(t_i) := \left\{ \begin{array}{lcr}
t_i & \text{if} & i <j\\
t_i + t_{i+1} & \text{if} & i =j \\
t_{i+1} & \text{if} & i > j \\
\end{array} \right.\,.
\end{eqnarray*} 
Following \cite[\S2.2.5]{book}, the {\em homotopy $K$-theory functor} is defined by the formula
\begin{eqnarray*}
KH\colon \dgcat(k) \too \Spt && \cA \mapsto \mathrm{hocolim}_{n\geq 0}\bbK(\cA\otimes \Delta_n)\,,
\end{eqnarray*}
where $\Spt$ stands for the homotopy category of spectra.
\section{Proof of Theorem \ref{thm:main1}}
Let us write $\cD^b(\mathrm{Mod}(X))$ for the bounded derived category of $\cO_X$-modules, $\cD^b(\mathrm{coh}(X))$ for the bounded derived category of coherent $\cO_X$-modules, and $\perf(X)$ for the category of perfect complexes of $\cO_X$-modules. These categories admit canonical dg enhancements $\cD_\dg^b(\mathrm{Mod}(X))$, $\cD^b_\dg(\mathrm{coh}(X))$, and $\perf_\dg(X)$, respectively; see \cite[\S4.4-\S4.6]{ICM-Keller}\cite{LO}. Recall from Orlov \cite{Orlov} that the dg category of singularities $\cD_\dg^{\mathrm{sg}}(X)$ is defined as the Drinfeld's dg quotient $\cD^b_\dg(\mathrm{coh}(X))/\perf_\dg(X)$. Consequently, we have the following short exact sequence of dg categories (see \cite[\S4.6]{ICM-Keller}):
$$
0 \too \perf_\dg(X) \too \cD^b_\dg(\mathrm{coh}(X)) \too \cD_\dg^{\mathrm{sg}}(X) \too 0\,.
$$
Since homotopy $K$-theory $KH$ is a localizing invariant of dg categories (see \cite[\S8.2.2]{book}) and $KH(\perf_\dg(X))$ is isomorphic to $KH(X)$ (see \cite[\S2.2.5]{book}), we obtain an induced distinguished triangle of spectra:
\begin{equation}\label{eq:triangle}
KH(X) \too KH(\cD^b_\dg(\mathrm{coh}(X))) \too KH(\cD_\dg^{\mathrm{sg}}(X)) \too \Sigma KH(X)\,.
\end{equation}
\begin{proposition}\label{prop:key1}
We have $KH_n(\cD^b_\dg(\mathrm{coh}(X)))=0$ for every $n <0$.
\end{proposition}
\begin{proof}
Given a dg category $\cA$ and an integer $p\geq 1$, let us write $\cA[t_1, \ldots, t_p]$ for the tensor product $\cA\otimes k[t_1, \ldots, t_p]$. Since homotopy $K$-theory is defined as the ``realization'' of a simplicial spectrum, we have the following standard convergent right half-plane spectral sequence :
\begin{equation}\label{eq:spectral-seq}
E^1_{pq} = N^p \bbK_q(\cA) \Rightarrow KH_{p+q}(\cA)\,.
\end{equation}
Here, $N^0\bbK_q(\cA):=\bbK_q(\cA)$ and $N^p\bbK_q(\cA), p \geq 1$, is defined as the intersection
$$\bigcap^p_{i=1} \mathrm{Kernel} \left(\bbK_q(\cA[t_1,\ldots, t_p]) \stackrel{\id \otimes (t_i=0)}{\too} \bbK_q(\cA[t_1, \ldots, \widehat{t_i}, \ldots, t_p])\right)\,.$$
Let us treat now the case where $\cA$ is the dg category $\cD^b_\dg(\mathrm{coh}(X))$. Using the fact that the dg category $k[t_1, \ldots, t_p]$ is Morita equivalent to $\perf(\mathrm{Spec}(k[t_1, \ldots, t_p]))\simeq \perf(\mathrm{Spec}(k[t_1]))\otimes \cdots \otimes \perf(\mathrm{Spec}(k[t_p]))$, Lemma \ref{lem:key} below (applied inductively) implies that $\cD^b_\dg(\mathrm{coh}(X))[t_1, \ldots, t_p]$ is Morita equivalent to $\cD^b_\dg(\mathrm{coh}(X[t_1,\ldots, t_p]))$. Consequently, since nonconnective algebraic $K$-theory is invariant under Morita equivalences, we obtain the following identifications
$$ \bbK(\cD^b_\dg(\mathrm{coh}(X))[t_1, \ldots, t_p]) \simeq \bbK(\cD^b_\dg(\mathrm{coh}(X[t_1, \ldots, t_p]))) = G(X[t_1, \ldots, t_p])\,,$$
where $G$ stands for $G$-theory; see \cite[Thm.~5.1]{ICM-Keller}. As proved by Schlichting in \cite[Thm.~7]{Schlichting}, the abelian groups $G_q(X[t_1, \ldots, t_p]), q<0$, are zero; this uses the fact that $X$, and hence $X[t_1, \ldots, t_p]$, is Noetherian. Therefore, we conclude that $N^p\bbK_q(\cD^b_\dg(\mathrm{coh}(X)))=0$ for every $p\geq 0$ and $q<0$. This implies that the spectral sequence \eqref{eq:spectral-seq} degenerates below the line $q=-1$ and consequently that
\begin{eqnarray*}
KH_n(\cD^b_\dg(\mathrm{coh}(X)))\simeq \bbK_n(\cD^b_\dg(\mathrm{coh}(X)))=G_n(X) && n <0\,.
\end{eqnarray*}
The proof follows now from the fact that $G_n(X)=0$ for every $n<0$.
\end{proof}
\begin{lemma}[Rouquier\footnote{Lemma \ref{lem:key} is a particular case of a general result of Rapha\"el Rouquier \cite{Rouquier}. The author is very grateful to Rapha\"el for kindly explaining him his ideas.}]\label{lem:key}
Given an affine $k$-scheme of finite type $X=\mathrm{Spec}(R)$, we have the following induced Morita equivalence:
\begin{eqnarray*}
\cD^b_\dg(\mathrm{coh}(X)) \otimes \perf_\dg(\mathrm{Spec}(k[t])) \too \cD^b_\dg(\mathrm{coh}(X[t])) && (\cF, \cF') \mapsto \cF\boxtimes \cF'\,.
\end{eqnarray*}
\end{lemma}
\begin{proof}
Recall first that $\cO_X$ and $\cO_{\mathrm{Spec}(k[t])}$ are generators of the triangulated categories $\perf(X)$ and $\perf(\mathrm{Spec}(k[t]))$, respectively. As proved by Rouquier in \cite[Thm.~7.38]{Rouquier1}, the triangulated category $\cD^b(\mathrm{coh}(X))$ admits a generator $\cG$. We claim that $\cG[t]:=\cG \boxtimes \cO_{\mathrm{Spec}(k[t])}$ is a generator of the triangulated category $\cD^b(\mathrm{coh}(X[t]))$. In order to prove this claim, we will proceed by induction on the dimension of $X$. The case $\mathrm{dim}(X)=0$ is clear. Assume then that the claim holds for every affine $k$-scheme of finite type of dimension $< \mathrm{dim}(X)$. Let us write $\cT$ for the smallest thick triangulated subcategory of $\cD^b(\mathrm{coh}(X[t]))$ containing the object $\cG[t]$. Given any object $\cH \in \cD^b(\mathrm{coh}(X[t]))$, we need to show that $\cH \in \cT$. Let us denote by $i\colon X^s \hookrightarrow X$ the singular locus of $X$ and by $\cD^b_{X^s[t]}(\mathrm{coh}(X[t]))$ the full triangulated subcategory of $\cD^b(\mathrm{coh}(X[t]))$ consisting of those bounded complexes of coherent $\cO_{X[t]}$-modules that are supported  on the closed subscheme $X^s[t]$. Similarly to the proof of \cite[Thm.~7.38]{Rouquier1}, we have a distinguished triangle 
$$\cH_1 \too \cH \oplus \Sigma \cH \too \cH_2 \too \Sigma \cH_1$$
with $\cH_1 \in \perf(X[t])$ and $\cH_2 \in \cD^b_{X^s[t]}(\mathrm{coh}(X[t]))$. Therefore, in order to prove our claim, it suffices to show that $\cH_1$ and $\cH_2$ belong to $\cT$. Since $\cG$ is a generator of $\cD^b(\mathrm{coh}(X))$, the object $\cO_{X[t]}=\cO_X \boxtimes \cO_{\mathrm{Spec}(k[t])}$ belongs to $\cT$. Consequently, since $\cO_{X[t]}$ is a generator of $\perf(X[t])$, the object $\cH_1$ also belongs to $\cT$. In what concerns the object $\cH_2$, let us start by choosing a generator $\cG'$ of the triangulated category $\cD^b(\mathrm{coh}(X^s))$. Since $\mathrm{dim}(X^s)< \mathrm{dim}(X)$, the induction assumption implies that $\cG'[t]:=\cG' \boxtimes \cO_{\mathrm{Spec}(k[t])}$ is a generator of the triangulated category $\cD^b(\mathrm{coh}(X^s[t]))$. Now, a proof similar to the one of \cite[Thm.~6.8(i)]{Gysin}, with \cite[Prop.~(19.1.1)]{EGA} replaced by the finite type assumption of $X$ and with \cite[Prop.~6.1]{Neeman} replaced by \cite[Prop.~6.6]{Rouquier}, shows that $i_\ast(\cO_{X^s}) \boxtimes \cO_{\mathrm{Spec}(k[t])}$ is a generator of the triangulated category $\cD^b_{X^s[t]}(\mathrm{coh}(X[t]))$. This implies that $i_\ast(\cG')[t]:=i_\ast(\cG') \boxtimes \cO_{\mathrm{Spec}(k[t])}$ is also a generator of $\cD^b_{X^s[t]}(\mathrm{coh}(X[t]))$. Since $i_\ast(\cG')$ belongs to the category $\cD^b(\mathrm{coh}(X))$ and $\cG$ is a generator of $\cD^b(\mathrm{coh}(X))$, the object $i_\ast(\cG')[t]$ belongs to $\cT$. Consequently, $\cH_2$ also belongs to $\cT$. This concludes the proof of our claim. 

Since $\cG$ is a generator of $\cD^b(\mathrm{coh}(X))$ and $\cG[t]$ a generator of $\cD^b(\mathrm{coh}(X[t]))$, we have induced Morita equivalences between dg categories
\begin{eqnarray*}
\cD^b_\dg(\mathrm{coh}(X)) \simeq \perf_\dg({\bf R}\mathrm{End}(\cG)) && \cD^b_\dg(\mathrm{coh}(X[t])) \simeq \perf_\dg({\bf R}\mathrm{End}(\cG[t]))\,,
\end{eqnarray*}
where ${\bf R}\mathrm{End}(\cG)$ and ${\bf R}\mathrm{End}(\cG[t])$ stand for the (derived) dg $k$-algebra of endomorphisms of $\cG$ and $\cG[t]$, respectively. Consider the following classical adjunction
\begin{equation}\label{eq:adjunction-last}
\xymatrix{
\cD^b(\mathrm{Mod}(X[t])) \ar@<1ex>[d]^-{\mathrm{Res}} \\
\cD^b(\mathrm{Mod}(X)) \ar@<1ex>[u]^-{\cF \mapsto \cF[t]}\,.
}
\end{equation}
By assumption, the affine $k$-scheme $X$ is of finite type and hence Noetherian. Therefore, since $\cG$ belongs to the triangulated subcategory $\cD^b(\mathrm{coh}(X))\subset \cD^b(\mathrm{mod}(X))$ and $\mathrm{Res}(\cG[t])\simeq \oplus^\infty_{i=1} \cG$, the combination of \cite[Cor.~6.17]{Rouquier1} with \eqref{eq:adjunction-last} implies that the dg $k$-algebra ${\bf R}\mathrm{End}(\cG[t])$ is quasi-isomorphic to ${\bf R}\mathrm{End}(\cG)[t]$. This concludes the proof of Lemma~\ref{lem:key} because $\perf_\dg(\mathrm{Spec}(k[t]))$ is Morita equivalent to $k[t]$.
\end{proof}
By combining Proposition \ref{prop:key1} with the long exact sequence of (stable) homotopy groups associated to the distinguished triangle of spectra \eqref{eq:triangle}, we conclude that $KH_n(\cD_\dg^{\mathrm{sg}}(X)) \simeq KH_{n-1}(X)$ for every $n<0$. Consequently, the proof of Theorem \ref{thm:main1} follows from the following result:
\begin{proposition}\label{prop:key2}
We have $KH_n(\cD_\dg^{\mathrm{sg}}(X))=0$ for every $n<0$.
\end{proposition}
\begin{proof}
Consider the polynomial algebra $S:=k[t_1, \ldots, t_d]$ as a $\bbZ$-graded $k$-algebra with $\mathrm{deg}(t_i)=1$. Note that since the $G$-action on $S$ preserves the $\bbZ$-grading, $R:=S^G$ is a $\bbZ$-graded $k$-subalgebra of $S$. Following Orlov \cite{Orlov}, in addition to the dg category of singularities $\cD^{\mathrm{sg}}_\mathrm{dg}(X)$, we can also consider the dg category of graded singularities $\cD^{\mathrm{sg}, \mathrm{gr}}_\mathrm{dg}(X)$. By construction, this latter dg category comes equipped with a degree shift dg functor $(1)\colon \cD^{\mathrm{sg}, \mathrm{gr}}_\mathrm{dg}(X) \to \cD^{\mathrm{sg}, \mathrm{gr}}_\mathrm{dg}(X)$. Following Keller \cite[\S7.2]{Orbit-Keller}, let us write $\cD^{\mathrm{sg}, \mathrm{gr}}_\mathrm{dg}(X)/(1)^\bbZ$ for the associated dg orbit category. As proved by Keller-Murfet-Van den Bergh in \cite[Props. A.2 and A.8]{KMV}, the forgetful dg functor  $\cD^{\mathrm{sg},\mathrm{gr}}_\mathrm{dg}(X)\to \cD^{\mathrm{sg}}_\mathrm{dg}(X)$ induces a Morita equivalence $\cD^{\mathrm{sg},\mathrm{gr}}_\mathrm{dg}(X)/(1)^\bbZ\simeq\cD^{\mathrm{sg}}_\mathrm{dg}(X)$. Consequently, since homotopy $K$-theory is an $\bbA^1$-homotopy invariant of dg categories (see \cite[\S8.5]{book}), \cite[Thm.~1.5]{Orbit} yields a distinguished triangle of spectra:
\begin{equation*}
KH(\cD^{\mathrm{sg},\mathrm{gr}}_\mathrm{dg}(X))\stackrel{KH((1))-\id}{\too} KH(\cD^{\mathrm{sg},\mathrm{gr}}_\mathrm{dg}(X)) \to KH(\cD^{\mathrm{sg}}_\mathrm{dg}(X)) \to \Sigma KH(\cD^{\mathrm{sg},\mathrm{gr}}_\mathrm{dg}(X))\,.
\end{equation*}
Let us denote by $\underline{\mathrm{MCM}}^{\mathrm{gr}}(R)$ the stable category of $\bbZ$-graded maximal Cohen-Macaulay $R$-modules. As proved by Iyama-Takahashi in \cite[Cor.~2.12]{IT}, the triangulated category $\underline{\mathrm{MCM}}^{\mathrm{gr}}(R)$ admits a full exceptional collection $(\cE_1, \ldots, \cE_v)$ with $D_r:=\mathrm{End}(\cE_r), 1\leq r \leq v$, finite dimensional division $k$-algebras. Making use of the classical equivalence of categories $\dgHo(\cD^{\mathrm{sg},\mathrm{gr}}_\mathrm{dg}(X))= \cD^{\mathrm{sg},\mathrm{gr}}(X) \simeq \underline{\mathrm{MCM}}^{\mathrm{gr}}(R)$ (see Buchweitz \cite{Buchweitz}), and of the fact that homotopy $K$-theory is an additive invariant of dg categories (see \cite[\S2]{book}), we hence conclude that $KH(\cD^{\mathrm{sg},\mathrm{gr}}_\mathrm{dg}(X))\simeq \oplus_{r=1}^v KH(D_r)$. Since $D_r$ is a (right) regular noetherian ring, we have $KH(D_r)\simeq \bbK(D_r)$ (see Gersten \cite[Prop.~3.14]{Gersten}) and $\bbK_n(D_r)=0$ for every $n <0$ (see Schlichting \cite[Thm.~7]{Schlichting}). Therefore, the proof follows now from the long exact sequence of (stable) homotopy groups associated to the above distinguished triangle of spectra. 
\end{proof}
\section{Proof of Proposition \ref{prop:1}}
Given any prime power $l^\nu$, we have the universal coefficient sequences:
\begin{equation}\label{eq:seq1}
0 \to KH_n(X) \otimes_\bbZ \bbZ/l^\nu \to KH_n(X; \bbZ/l^\nu) \to \{ l^\nu\text{-}\mathrm{torsion}\,\mathrm{in}\,KH_{n-1}(X)\} \to 0  
\end{equation}
\begin{equation}\label{eq:seq2}
0 \to \bbK_n(X) \otimes_\bbZ \bbZ/l^\nu \to \bbK_n(X; \bbZ/l^\nu) \to \{ l^\nu\text{-}\mathrm{torsion}\,\mathrm{in}\,\bbK_{n-1}(X)\} \to 0  \,.
\end{equation}
Theorem \ref{thm:main1}, combined with \eqref{eq:seq1}, implies that $KH_n(X;\bbZ/l^\nu)=0$ for every $n <-1$. Since the base field $k$ is of characteristic zero, $KH_n(X;\bbZ/l^\nu)$ is isomorphic to $\bbK_n(X;\bbZ/l^\nu)$; see Weibel \cite[Page 391]{Weibel1}. Therefore, it follows from \eqref{eq:seq2} that $\bbK_{-2}(X)$ is $l^\nu$-divisible and that $\bbK_n(X), n <-2$, is uniquely $l^\nu$-divisible. The proof follows now from the fact that the prime power $l^\nu$ is arbitrary.
\section{Proof of Theorem \ref{thm:main2}}
The triangle of spectra \eqref{eq:triangle} yields a long exact sequence of abelian groups 
$$
\cdots \too KH_0(\cD^{\mathrm{sg}}_\dg(X)) \too KH_{-1}(X) \too KH_{-1}(\cD^b_\dg(\mathrm{coh}(X))) \too \cdots
$$
Thanks to Proposition \ref{prop:key1}, we have $KH_{-1}(\cD^b_\dg(\mathrm{coh}(X)))=0$. Therefore, it follows that $KH_{-1}(X)$ is a quotient of $KH_0(\cD_\dg^{\mathrm{sg}}(X))$. Let us now compute this latter group. As explained in the proof of Proposition \ref{prop:key2}, $KH(\cD^{\mathrm{sg},\mathrm{gr}}_\mathrm{dg}(X))$ is isomorphic to the direct sum $\oplus_{r=1}^v KH(D_r)$. Since by assumption $k$ is algebraically closed, all the finite dimensional division $k$-algebras $D_r$ are isomorphic to $k$. Therefore, since $KH_0(k)\simeq \bbK_0(k)\simeq \bbZ$ and $KH_{-1}(k)\simeq \bbK_{-1}(k)=0$, we conclude from the long exact sequence of (stable) homotopy groups associated to the distinguished triangle of spectra constructed in the proof of Proposition \ref{prop:key2} and from the natural  identification $K_0(\cD^{\mathrm{sg},\mathrm{gr}}_\mathrm{dg}(X))=K_0(\cD^{\mathrm{sg},\mathrm{gr}}(X))$, that the group $KH_0(\cD^{\mathrm{sg}}_\dg(X))$ is isomorphic to the cokernel of the induced homomorphism:
\begin{equation}\label{eq:homo-1}
K_0((1))-\id \colon K_0(\cD^{\mathrm{sg},\mathrm{gr}}(X))\too K_0(\cD^{\mathrm{sg},\mathrm{gr}}_\mathrm{dg}(X))\,.
\end{equation}
Via the equivalence of categories $\cD^{\mathrm{sg},\mathrm{gr}}(X)\simeq \underline{\mathrm{MCM}}^{\mathrm{gr}}(R)$, \eqref{eq:homo-1} reduces to
\begin{equation}\label{eq:homo-2}
K_0((1))-\id \colon K_0(\underline{\mathrm{MCM}}^{\mathrm{gr}}(R))\too K_0(\underline{\mathrm{MCM}}^{\mathrm{gr}}(R))\,.
\end{equation}
As proved by Amiot-Iyama-Reiten in \cite[Thm.~4.1]{AIR}, the triangulated category $\underline{\mathrm{MCM}}^{\mathrm{gr}}(R)$ admits a tilting object $T$. Moreover, the associated $k$-algebra $A:=\mathrm{End}(T)$ is finite dimensional and of finite global dimension. Let us denote by $\cD^b(\mathrm{mod}(A))$ the bounded derived category of finitely generated (right) $A$-modules. By construction, this latter category comes equipped with the Serre functor $\mathrm{S}$ and with the Auslander-Reiten translation functor $\tau$. As proved in {\em loc. cit.}, via the equivalence of categories $\underline{\mathrm{MCM}}^{\mathrm{gr}}(R) \simeq \cD^b(\mathrm{mod}(A))$ induced by the tilting object $T$, the degree shift functor $(1)$ corresponds to the functor $\mathrm{S}^{-1}\Sigma^{d-1}$. Since $\mathrm{S}^{-1}\Sigma=\tau$, we hence conclude that the above homomorphism \eqref{eq:homo-2} reduces to 
\begin{equation}\label{eq:homo-3}
(-1)^{d-2}\Phi_A-\id \colon K_0(\cD^b(\mathrm{mod}(A))) \too K_0(\cD^b(\mathrm{mod}(A)))\,,
\end{equation}
where $\Phi_A$ stands for the inverse of the Coxeter matrix of $A$. As proved in \cite[\S5]{AIR}, the $k$-algebra $A$ is isomorphic to the $k$-algebra $kQ/\langle \rho\rangle$ associated to the quiver with relations $(Q,\rho)$ introduced at \S\ref{sec:introduction}. On the one hand, this implies that the number of simple (right) $A$-modules agrees with the number of vertices of $Q$, \ie $v=m-1$. On the other hand, this implies that the matrix $\Phi_A$ can be written as $-C(C^{-1})^T$, where $C_{ij}$ equals the number of arrows in $Q$ from $j$ to $i$ (counted modulo the relations). Consequently, the above homomorphism \eqref{eq:homo-3} reduces to the (matrix) homomorphism $\mathrm{M}\colon \oplus^{m-1}_{r=1} \bbZ \to \oplus^{m-1}_{r=1} \bbZ$ introduced at \S\ref{sec:introduction}. This concludes the proof.

\medbreak\noindent\textbf{Acknowledgments:}
This work was conceived during a visit to Christian Haesemeyer at the University of Melbourne. The author is very grateful to Christian for stimulating discussions, for his interest on these results, and for its kind hospitality. The author also would like to thank Guillermo Corti\~nas for his interest on these results and for correcting a typo in a previous version of this article.

\end{document}

\end{proof}